\documentclass{article} % For LaTeX2e
\usepackage{nips14submit_e,times}
\usepackage{hyperref}
\usepackage{url}
\usepackage{wrapfig}

\title{\Large{Subsampling Methods for Persistent Homology}}

\author{
\vspace{-2em} \\
\begin{tabular}{l}
\textbf{Fr\'ed\'eric Chazal}, \texttt{frederic.chazal@inria.fr}
\\ \textbf{Brittany Terese Fasy}, \texttt{brittany.fasy@alumni.duke.edu}
\\ \textbf{Fabrizio Lecci}, \texttt{lecci@cmu.edu}
\\ \textbf{Bertrand Michel}, \texttt{bertrand.michel@upmc.fr}
\\ \textbf{Alessandro Rinaldo}, \texttt{arinaldo@stat.cmu.edu}
\\ \textbf{Larry Wasserman}, \texttt{larry@stat.cmu.edu} 
\end{tabular}
\vspace{-1em}
}
\nipsfinalcopy % Uncomment for camera-ready version

\usepackage{amsmath,amssymb,graphicx}
\usepackage[authoryear]{natbib}
\usepackage{amsthm}
\usepackage{algorithm}% http://ctan.org/pkg/algorithms
\usepackage{algpseudocode}% http://ctan.org/pkg/algorithmicx
\usepackage{cleveref}
\usepackage{bbm} % for indicator function \mathbbm{1}

\bibpunct{(}{)}{;}{a}{,}{,}

\newtheorem{theorem}{Theorem}
\newtheorem{lemma}[theorem]{Lemma}

\newtheorem{corollary}[theorem]{Corollary}
\newtheorem{proposition}[theorem]{Proposition}
\newtheorem{remark}[theorem]{Remark}
\theoremstyle{definition}
\newtheorem{definition}[theorem]{Definition}

\newcommand{\figref}[1]{Figure~\ref{#1}}

\newcommand{\thmref}[1]{Theorem~\ref{#1}}

\usepackage{color}

\algnewcommand\algorithmicinput{\textbf{INPUT:}}
\algnewcommand\INPUT{\item[\algorithmicinput]}
\algnewcommand\algorithmicoutput{\textbf{OUTPUT:}}
\algnewcommand\OUTPUT{\item[\algorithmicoutput]}

\newcommand{\E}{\mbox{$\mathbb{E}$}}

\newcommand{\R}{\mbox{$\mathbb{R}$}}
\newcommand{\X}{\mbox{$\mathbb{X}$}}
\newcommand{\Z}{\mbox{$\mathbb{Z}$}}

\newcommand{\K}{\mathcal{C}}  % simplicial complex
\newcommand{\Cech}{\operatorname{\mathrm{Cech}}}
\newcommand{\Rips}{\operatorname{\mathrm{Rips}}}
\newcommand{\Filt}{\operatorname{\mathrm{Filt}}}

\newcommand{\metricfont}{\mathrm}
\newcommand{\bottle}{\metricfont{d_b}}
\newcommand{\dgh}{\metricfont{d_{\textrm{\tiny GH}}}}
\newcommand{\dhaus}{H}

\usepackage[font=small]{caption}

\newcommand{\dgm}{D} % persistence diagram
\newcommand{\dgmspace}{\mathcal{D}_T} % space of T-bounded persistence diagrams
\newcommand{\lscape}{\lambda} % persistence diagram
\newcommand{\lscapespace}{\mathcal{L}_T} % space of landscapes
\let\hat\widehat
\let\tilde\widetilde

\begin{document}

\maketitle
\begin{abstract}
Persistent homology is a multiscale method for analyzing the shape of sets and
functions from point cloud data arising from an unknown distribution supported
on those sets.
%Given data from a distribution supported on a set,
%it is possible to estimate the persisent homology from the data.
When the size of the sample is large, direct computation of the persistent
homology is prohibitive due to the combinatorial nature of the existing algorithms.
%However, computing this estimate is
%computationally expensive.
We propose to compute the persistent homology of several subsamples
of the data and then combine the resulting estimates. 
We study the risk of two estimators 
and we prove that the subsampling approach carries stable
topological information while achieving a great reduction in computational
complexity.
\end{abstract}

%----------------------------------------------------------
\section{Introduction}
%----------------------------------------------------------

Topological Data Analysis (TDA) refers to a
collection of methods for finding topological structure in data
\citep{carlsson2009topology}.
The input is a dataset drawn from a probability measure
supported on
an unknown low-dimensional set
$\mathbb{X}$. The
output is a collection of data summaries that are used to estimate the
topological features of $\mathbb{X}$.

%These data summaries have been used successfully in a variety of
%applied problems, ranging from medical imaging and neuroscience
%\citep{singh2008topological, chung2009persistence,
%pachauri2011topology} to bioinformatics
%\citep{kasson2007persistent}, cosmology \citep{sousbie2011persistent,
%van2011alpha, cisewski2014nonparametric}, sensor networks
%\citep{de2007coverage}, landmark-based shape data analyses
%\citep{gamble2010exploring}, shape classification
%\citep{chazal2009gromov}, and clustering
%\citep{chazal2013persistence}.  Nonetheless, the statistical
%properties of the data summaries produced in TDA and, more generally,
%of the usually heuristic data-analytic methods they are part of, have
%remained largely unexplored.

One approach to TDA is persistent homology
\citep{edelsbrunner2010computational}, which is a method for studying the
topology at
multiple scales simultaneously.
For example,
let $A$ be a set and let
$f(x) = \inf_{y\in A}||x-y||$
be the {\em distance function.}
The lower-level sets $\{x: f(x) \leq t \}$ change as $t$ increases from
$-\infty$ to $\infty$.
Persistent homology summarizes the evolution of
$\{x: f(x) \leq t \}$ as a function of $t$.
In particular, the {\em persistence diagram}
represents the birth and death time of each topological feature
as a point in the plane.
Thanks to stability properties \citep{cohen2007stability, chazal2009proximity,
chazal2012structure, cso-psgc-12}, persistence diagrams provide relevant
multi-scale topological information about the data; see
Section \ref{sec:background}.
The persistence diagram can be converted into a summary function
called a {\rm landscape} \citep{bubenik2012statistical}.

\textbf{Contribution and Related Work.}
The time and space complexity
of persistent homology algorithms is one of the main obstacles in applying TDA
techniques to high-dimensional problems.
To overcome the problem of computational costs, we propose the following
strategy:
given a large point cloud, take several subsamples, compute the landscape for
each subsample
and then combine the information.
More precisely, let $\lambda$ be a random persistence landscape from
$\Psi_{\mu}^m$, a measure on the space of landscape functions induced by a
sample of size $m$ from a metric measure space $(\mathbb{X}, \rho, \mu)$. We
show that the average landscape is stable with respect to perturbations of the
underlying measure $\mu$ in the Wasserstein metric; see~\thmref{thm:stabW}.
% (defined later):
% $\left\Vert \mathbb{E}_{\Psi_\mu^m}[\lambda] - \mathbb{E}_{\Psi_\nu^m}[\lambda']
% \right\Vert_\infty \leq  m^{\frac{1}{p}} W_{\rho,p}(\mu,\nu).$
The empirical counterpart of the average landscape is
$\overline{\lambda_n^m}= \frac{1}{n} \sum_{i=1}^n \lambda_i,$
where $\lambda_1, \dots, \lambda_n \sim \Psi_{\mu}^m$. The empirical average
landscape can be used as an unbiased estimator of
$\mathbb{E}_{\Psi_\mu^m}[\lambda]$ and as a biased estimator of
$\lambda_{\mathbb{X}_\mu}$, the computationally expensive persistence landscape
associated to the support of the measure $\mu$. Unlike
$\lambda_{\mathbb{X}_\mu}$,
the estimator $\overline{\lambda_n^m}$ is robust to the presence of
outliers.
In the same spirit, we propose a different estimator constructed by choosing a
sample of $m$ points of $\mathbb{X}$ as close as possible to $\mathbb{X}_\mu$,
and then computing its persistent homology to
approximate~$\lambda_{\mathbb{X}_\mu}$. See
Section \ref{sec:subsampling} for more details.
%In symbols, let $S_1^m, \dots, S_n^m$ be $n$ independent samples of size $m$
%from $\mu$.
%The closest sample is
%$$
%\hat{C^m_n}= \arg\min_{S \in \{S_1^m, \dots, S_n^m\}} H(S, X_\mu)
%$$
%and the corresponding landscape function is
%$\hat{\lambda_n^m} = \lambda_{D_{\hat{C^m_n}}}. $

Closely related to our approach, the distribution of persistence diagrams
associated to subsamples of fixed size has also been proposed in
\cite{blumberg2012persistent}.
There, the authors show that the distribution of persistence diagrams associated
to subsamples
of fixed size is stable with respect to perturbations of the underlying measure
in the Gromov-Prohorov
metric. Though similar in spirit, our approach relies on different techniques
and, in particular,
leads to easily computable summaries of the persistent homology of a given
space. 
These summaries are particularly useful when the exact computation of the persistent homology 
is unfeasbile, as in the case of large point clouds.

\textbf{Software.}
%There is a growing list of software available to compute persistent homology.
%\textsf{\textbf{javaPlex}}\footnote{https://code.google.com/p/javaplex/} is a
%Java software package developed by the Computational Topology workgroup at
%Stanford University;
%\textsf{\textbf{phom}}\footnote{
%http://cran.r-project.org/web/packages/phom/index.html} is an R package written
%by Andrew Tausz;
%\textsf{\textbf{PHAT}}\footnote{https://code.google.com/p/phat/} is a C++
%library developed at the Institute of Science and Technology Austria;
%\textsf{\textbf{Dionysus}}\footnote{http://www.mrzv.org/software/dionysus/} is
%a
%C++ library written and maintained by Dmitriy Morozov;
%finally, \textsf{\textbf{GUDHI}}\footnote{https://project.inria.fr/gudhi/} is
%new born project hosted by INRIA and whose goal is the development and
%implementation of new algorithms for geometric understanding in high
%dimensions.
%Preliminary results show that \textsf{\textbf{GUDHI}} outperforms its major
%competitors \textsf{\textbf{PHAT}} and \textsf{\textbf{Dionysus}}; see
%\cite{boissonnat2013compressed}.
We plan to release the R package \textsf{\textbf{persistence}},
which provides efficient algorithms for the computation of persistent homology
from \textsf{\textbf{Dionysus}} and \textsf{\textbf{GUDHI}}, and makes them
available with the user-friendly R interface.
\textsf{\textbf{Dionysus}}\footnote{http://www.mrzv.org/software/dionysus/} is a
C++ library written and maintained by Dmitriy Morozov;
\textsf{\textbf{GUDHI}}\footnote{https://project.inria.fr/gudhi/} is
new born project hosted by INRIA and whose goal is the development and
implementation of new algorithms for geometric understanding in high dimensions.
Preliminary results show that \textsf{\textbf{GUDHI}} outperforms its major
competitors; see \cite{boissonnat2013compressed}.
Our package includes a series of tools for the statistical analysis of
persistent homology, including the methods described in
\cite{Fasy2013statistical}, \cite{chazal2013stochastic}, and
this paper.

%Statistical methods for persistent homology provide an alternative to its exact
%computation. Knowing with high confidence that an approximated persistence
%diagrams is close to the true--computationally infeasible--diagram is often
%enough for practical purposes. In this spirit, \cite{Fasy2013statistical}
%develop several methods to find confidence sets for persistence diagrams,
%\cite{chazal2013optimal} derive optimal rates of convergence for estimated
%persistence diagrams, \cite{bubenik2012statistical} and
%\cite{chazal2013stochastic} introduce and analyze the persistence landscape, a
%real valued function that further summarizes the information contained in a
%persistence diagram.

\textbf{Outline.}
Background on persistent homology is presented in Section
\ref{sec:background}. Our approach is introduced in
Section \ref{sec:subsampling}, with a formal definition of the estimators briefly described in this
introduction.  Section \ref{sec:average} contains the stability result
of the average landscape.  Section \ref{sec:risk} is devoted to the
risk analysis of the proposed estimators.  In Section
\ref{sec:experiments}, we apply our methods to two examples.  We
conclude with some remarks in Section \ref{sec:conclusion} and defer
proofs and technical details to the appendices.

%----------------------------------------------------------
\section{Background}
%----------------------------------------------------------
\label{sec:background}

Computing persistent homology requires building
a nested sequence of geometric complexes indexed by a real parameter.
In this section, we briefly introduce these families and topological summaries
of them, but refer the reader
to Section 4.2. of~\cite{cso-psgc-12} for a complete definition of these
geometric filtered complexes and their use in TDA, to
\cite{edelsbrunner2010computational} for the definition of persistence
diagrams, and to \cite{bubenik2012statistical} for the definition and properties
of
persistence landscapes.

\subsection{Geometric Complexes}
\begin{figure}[ht!b!]
\centering
\includegraphics[height=2cm]{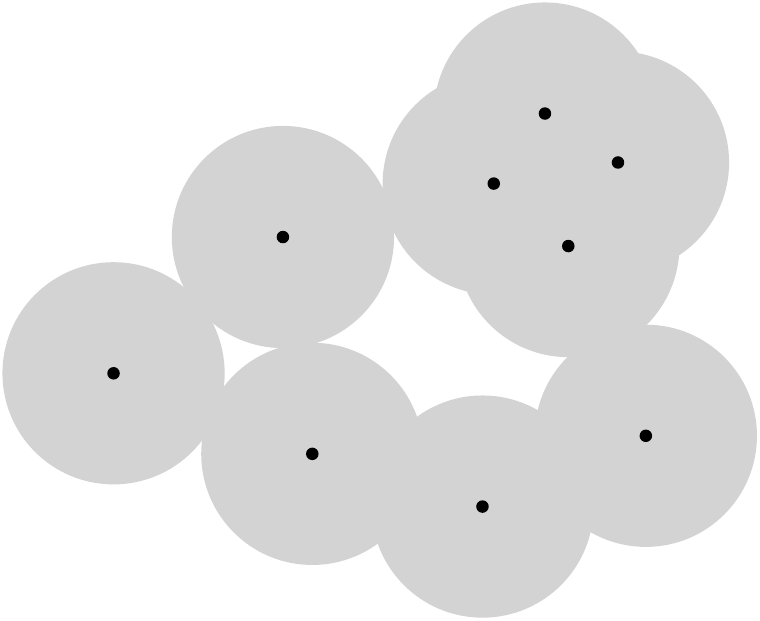}\ \
\includegraphics[height=2cm]{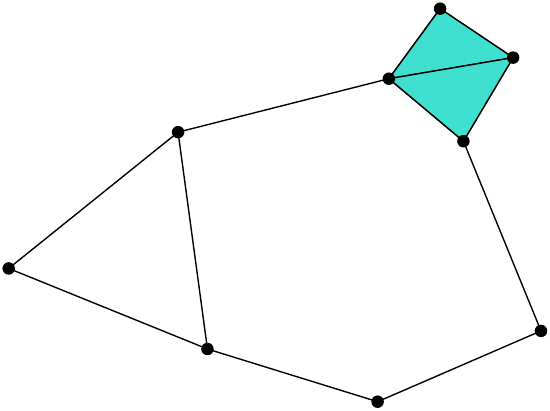}\ \
\includegraphics[height=2cm]{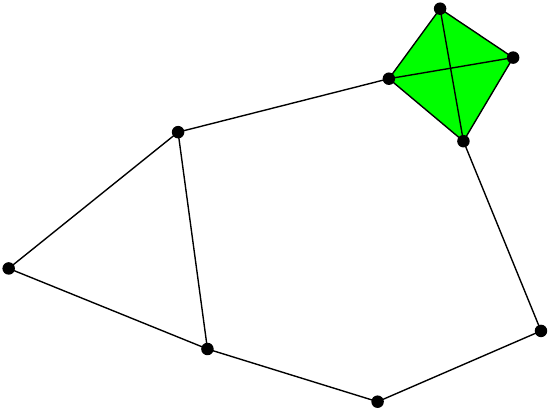}\ \
\includegraphics[height=2cm]{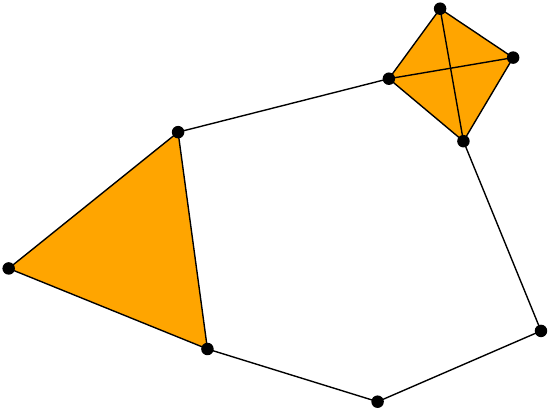}
\caption{From left to right: the $\alpha$ sublevelset of the distance function
to a point
  set $\X$ in $\R^2$, the $\alpha$-complex, $\Cech_\alpha(\X)$, and
$\Rips_{2\alpha}(\X)$.
The last two complexes include a tetrahedron.}
\label{fig:complex}
\end{figure}

To compute the persistent homology from a set of data,
we need to construct a set of structures called simplicial complexes.
A simplicial complex $\K$
is a set of simplices (points, segments, triangles, etc) such that any face
from a simplex in $\K$ is also in $\K$ and the intersection of any two
simplices of $\K$ is a (possibly empty) face of these simplices.

Given a metric space $\X$, we define three simplicial complexes whose vertex set is $\X$;
%that use the zero-skeleton $\{ x_i \}_{x_i \in \X}$; 
see Figure~\ref{fig:complex} for illustrations.  The
\emph{Vietoris-Rips complex} $\Rips_{\alpha}(\X)$ is the set of
simplices~$[x_0,\ldots,x_k]$ such that $d_{\X}(x_i,x_j)\leq \alpha$ for all
$(i,j)$.
The \emph{\v Cech complex} $\Cech_\alpha(\X)$ is similarly defined as
the set of simplices $[x_0,\ldots,x_k]$ such that there exists a point $x \in
\X$ for which $d_{\X}(x,x_i)\leq \alpha$ for all $i$.
%\footnote{When $\X$ is
%embedded in $\R^d$, we can extend
%the definition of the \v Cech complex to allow for $x \in \R^d$ instead of $x
%\in \X$.}
Note that these two complexes are related by
$\Rips_\alpha(\X)\subseteq\Cech_\alpha(\X)\subseteq\Rips_{2\alpha}(\X)$ and that their definition does not require $\X$ to be finite. 
% Note that these two families of complexes only depend on the pairwise
%distances between the points of $\X$.
When $\X \subset \R^d$,  we also define the \emph{$\alpha$-complex} as the set of simplices $[x_0,\ldots,x_k]$ such that there exists a ball of radius at most $\alpha$ containing $ x_0,\ldots,x_k$ on its boundary and whose interior does not intersect $\X$.
% We also define the
%\emph{$\alpha$-complex} as
%the set of simplices $[x_0,\ldots,x_k]$ such that the smallest circle defined
%by the set of vertices $\{x_0, \ldots, x_k  \}$ contains no point of $\X$ in
%the interior of the circle.
%While the Rips and the {\v C}ech complexes can have
%simplicies of dimension equal to the cardinality of $\X$, the
%$\alpha$-complex
%is a subcomplex of the
%Delaunay triangulation and thus only contains simplices of dimension at
%most~$d$.

%The {\v C}ech and $\alpha$-complexes have the same homology as the union of
%the balls $\bigcup_{x \in \X}
%B(x,\alpha)$.  Moreover, this union of the balls is also the
%$\alpha$-sublevel set of the
%distance function to $\X$, $d(\cdot,\X)$, and, as a consequence,
%increasing $\alpha$ from zero to infinity is a convenient way to study the
%evolution of the topology of the sublevel sets of $d(\cdot,\X)$.  We are
%interested in this evolution, as we can
%interpret a sublevel set of $d(\cdot, \X)$ as an estimate of the support of the
%distribution from which the points in $\X$ were drawn.

Each family described above is non-decreasing with
$\alpha$: for any $\alpha\leq\beta$, there is an inclusion of
$\Rips_\alpha(\X)$ in $\Rips_\beta(\X)$, and similarly for the \v Cech
and Alpha complexes. These sequences of inclusions are called
\emph{filtrations}.
In the following, we let $\Filt(\X) := (\Filt_\alpha (\X))_{\alpha
\in \mathcal A}$  denote  a filtration corresponding to one of the
parameterized complexes defined above.
%There are several other families that we
%could also have considered, most
%notably witness complexes, which were defined in \cite{cso-psgc-12}. Extending
%our results to them is
%straightforward and yields very similar results, so we will restrict to the
%families defined above in the rest of the paper.

\subsection{Persistence Diagrams}
The topology of $\Filt_\alpha (\X)$ changes as $\alpha$ increases: new connected components can appear, existing connected components can merge, cycles and cavities can appear or be filled, etc.
Persistent homology tracks these changes, identifies \emph{features} and associates an \emph{interval} or \emph{lifetime} (from $b$ to $d$) to them. For instance, a connected component is a feature that is born at the smallest $\alpha$ such that the component is present in $\Filt_\alpha (\X)$, and dies when it merges with an older connected component. 
Intuitively, the longer a feature persists, the more relevant it is.
The lifetime of a feature can be represented as a point in the plane with coordinates $(b,d)$.
%where the $x$-coordinate indicates the birth time and the $y$-coordinate the death time.
The obtained set of points (with multiplicity) is called the {\em persistence diagram} $\dgm(\Filt(\X))$ (and we will abuse terminology slightly by denoting it $D_{\X}$). Note that
the diagram is entirely contained in the half-plane above the diagonal
$\Delta$ defined by $y=x$, since death always occurs after birth.
\cite{chazal2012structure} shows that this diagram is still well-defined under very weak hypotheses,
%even in cases where the sequence might not be decomposable as a finite sum of intervals,
and in particular $\dgm(\Filt(\X))$ is well-defined for any
compact metric space $\X$ \cite{cso-psgc-12}.
The most persistent features (supposedly the most important) are those represented by the points furthest from the diagonal in the diagram, whereas points close to the diagonal can be interpreted as (topological) noise.

To avoid (minor) technical difficulties, we restrict our attention to
diagrams $D$ such that $(b,d) \in [0,T] \times [0,T]$ for all $(b,d) \in D$,
for some fixed $T>0$. Note that, in our setting, $D_{\X}$ is in  $\mathcal{D}_T$ as soon as $T$ is larger than the diameter of $\X$.
We denote by $\mathcal{D}_T$ the space of all such (restricted) persistence~diagrams and we endow it with a metric called the
\emph{bottleneck distance} $\bottle$. Given two persistence diagrams, it is
defined as the infimum of the $\delta$ for which we can find a matching between
the diagrams, such that two points can only be matched if their distance is less
than $\delta$ and all points at distance more than $\delta$ from the diagonal
must be matched.
%Note that points close to the diagonal $\Delta$ are easily ignored, which fits with their interpretation as irrelevant noise.
%%%%%%%%%%%%%%%%%%%%%%
%%%%%%%%%%%%%%%%%%%%%%

A fundamental property of persistence diagrams, proven in
\cite{chazal2012structure}, is their \emph{stability}.
Recall that the Hausdorff distance between two compact subsets $X,Y$ of a metric space $(\mathbb{X},\rho)$ is
$\displaystyle
\dhaus(X,Y)= \max \Big\{ \max_{x \in X} \min_{y \in Y} \rho(x,y) , \;  \max_{y
\in Y}
\min_{x \in X} \rho(x,y) \Big\}.
$
%Given a metric space $(\mathbb{X},\rho)$, the Hausdorff distance is
%defined as
%$\displaystyle
%\dhaus(X,Y)= \max \Big\{ \max_{x \in X} \min_{y \in Y} \rho(x,y) , \;  \max_{y
%\in Y}
%\min_{x \in X} \rho(x,y) \Big\},
%$
%for any $X,Y \subset \mathbb{X}$.
If $\X$ and $\tilde \X$
are two compact metric spaces, then one has
\begin{equation} \label{bot-dgh}
\bottle(
\dgm_{\X},
\dgm_{\tilde \X}
)
~\leq~
2 \dgh ( \X, \tilde \X ),
\end{equation}
where $\dgh  (\X, \tilde \X )$ denotes the Gromov-Hausdorff distance,
i.e., the
infimum Hausdorff distance between $\X$ and $\tilde \X$
over all possible isometric embeddings into a common metric space.
If $\X$ and $\tilde \X$ are already embedded in the same metric space then (\ref{bot-dgh}) holds for $\dhaus(\cdot,
\cdot)$ in place of $2\dgh(\cdot,\cdot)$.
% Note that these properties are intrinsic to the metrics defined on $\X$ and
% $\tilde \X$; they do not involve any probability
% measure.  Later in this paper, we will be interested in understanding the
% persistent homology of the support of a probability measure.

%\begin{figure}[ht!b!] 
%\centering
%\includegraphics[height=1.5in]{figs/triangles3b}\ \
%\includegraphics[height=1.5in]{figs/triangles-CI}
%\caption{The pink circles represent the points in a persistence diagram $D$.
%          As shown on the left, the landscape $\lscape(k, \cdot)$ is the $k$-th
%          largest
%          of the arrangement of the graphs of $\{ \Lambda_p \}$.
%          In particular, the cyan curve is the landscape $\lscape(1,\cdot)$.
%          On the right, we see a confidence interval for this landscape.}
%\label{fig:landscape}
%\end{figure}

\subsection{Persistence Landscapes}

The persistence landscape, introduced in \cite{bubenik2012statistical},
is a collection of continuous, piecewise linear
functions~\mbox{$\lscape \colon  \Z^{+} \times \R \to \R$} that summarizes a
persistence diagram.
To define the landscape, consider the set of functions created by tenting
each each point $p=(x,y)= \left( \frac{b+d}{2}, \frac{d-b}{2}  \right)$
representing a birth-death pair $(b,d) \in D$ as follows:
\begin{equation}\label{eq:triangle}
 \Lambda_p(t) =
 \begin{cases}
  t-x+y & t \in [x-y, x] \\
  x+y-t & t \in (x,  x+y] \\
  0 & \text{otherwise}
 \end{cases}
 =
 \begin{cases}
  t-b & t \in [b, \frac{b+d}{2}] \\
  d-t & t \in (\frac{b+d}{2}, d] \\
  0 & \text{otherwise}.
 \end{cases}
\end{equation}
%Notice that $p$ is itself on the graph of $\Lambda_p(t)$.

\begin{wrapfigure}{L}{0.5\textwidth}
% \vspace{-10pt}
\centering
	\includegraphics[height=1.5in]{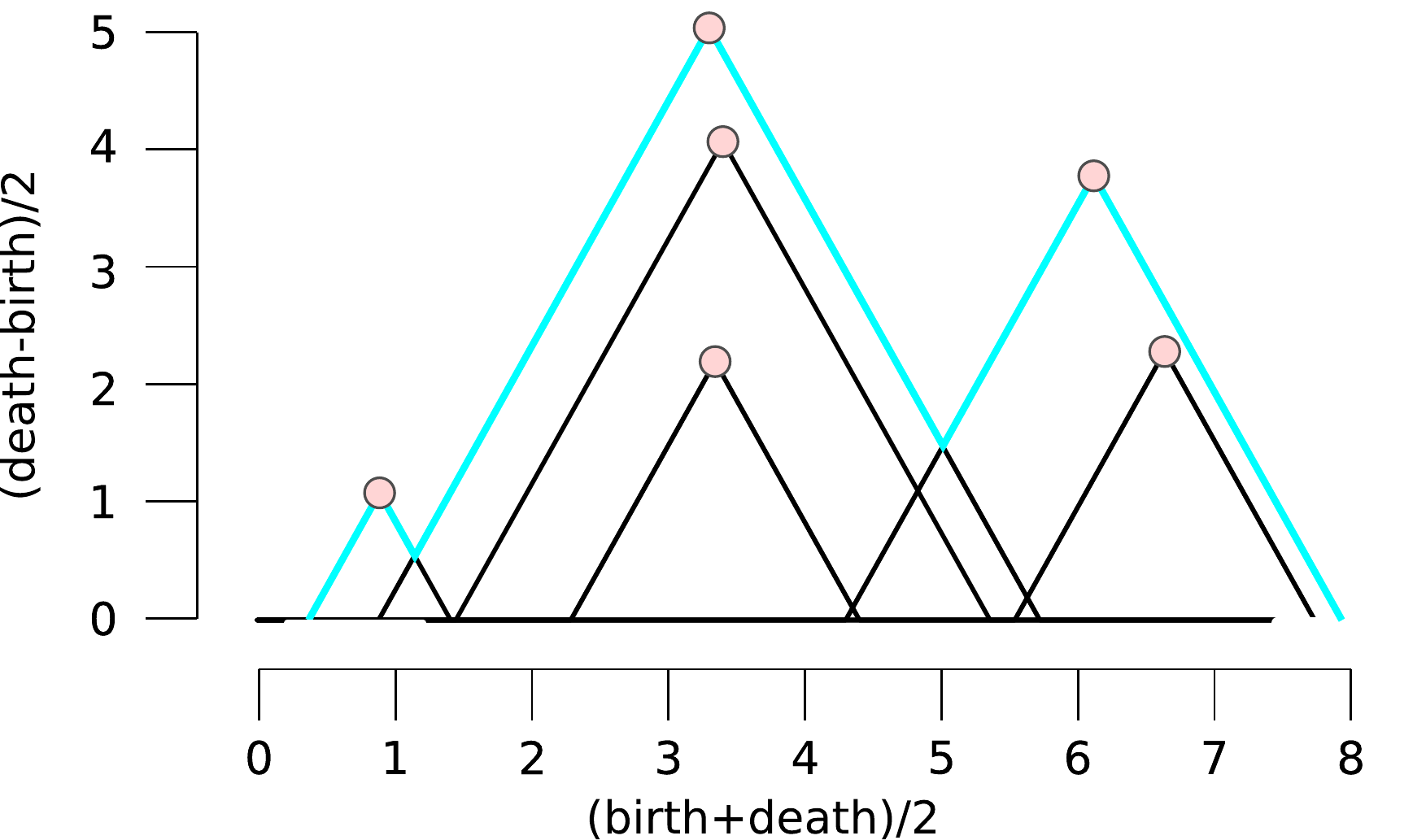}\ \
\caption{
	We use the rotated axes to represent a persistence diagram $D$.
	A feature $(b,d) \in D$ is represented by the point
	$(\frac{b+d}{2},\frac{d-b}{2})$ (pink).
	In words,
	the $x$-coordinate is the average parameter value over which the feature
exists,
	and the $y$-coordinate is the half-life of the feature.
        The cyan curve is the landscape $\lscape(1,\cdot)$.}
\label{fig:landscape}
 \vspace{-15pt}
\end{wrapfigure}

We obtain
an arrangement of piecewise linear curves by overlaying the graphs of
the functions~\mbox{$\{ \Lambda_p \}_{p}$}; see~\figref{fig:landscape}.
The persistence landscape of $\dgm$ is a  summary of this arrangement.
Formally, the persistence landscape of $D$ is the collection of functions
\begin{equation}\label{eq:landscape}
 \lscape_{\dgm}(k,t) = \underset{p}{\text{kmax}} ~ \Lambda_p(t), \quad
 t \in [0,T], k \in \mathbb{N},
\end{equation}
where kmax is the $k$th largest value in the set; in particular,
$1$max is the usual maximum~function. We set $\lscape_{\dgm}(k,t) = 0$ if the
set $\{ \Lambda_p(t)\}_p$ contains less than
$k$ points.  For simplicity of exposition, if $D_{\X}$ is the persistence
diagram of some metric space $\X$, then we use
$\lambda_{\X}$ to denote $\lambda_{D_{\X}}$.

We denote by $\lscapespace$ the space
of persistence landscapes corresponding to
$\dgmspace$.
From the definition of persistence landscape, we immediately observe that
$\lscape_{\dgm}(k,\cdot)$ is one-Lipschitz.
%,  since the composite functions $\Lambda_p$ are one-Lipschitz.
%  Hence, for all $\lambda \in \lscapespace$, $\lambda$
%is one-Lipschitz and vanishes outside of $(0,T)$; therefore, the landscape
%$\lambda$ is upper bounded by $T/2$.
The following additional properties are proven in \cite{bubenik2012statistical}.

\begin{lemma} \label{lemma:basic-prop}
Let $D, D'$ be persistence diagrams. We have the following for any $t\in
\mathbb{R}$ and any $k \in \mathbb{N}$:\\
(i)  $\lambda_D(k,t)
~\geq~ \lambda_D(k+1,t) ~\geq~ 0 $.\\
(ii) $| \lambda_D(k,t) - \lambda_{D'}(k,t)| ~\leq~  \bottle(D,D')$.
\end{lemma}
For ease of exposition, we focus on the case $k=1$,
and set $\lscape_D(t) = \lscape_{\dgm}(1,t)$.  However,
the results we present hold for~$k > 1$.  In fact, the results hold for more
general summaries of persistence landscapes, including the silhouette defined
in \cite{chazal2013stochastic}. %\\ % for spacing with the wrapfigure

%----------------------------------------------------------
\section{The Multiple Samples Approach}
%----------------------------------------------------------
\label{sec:subsampling}

Let $(\mathbb{X}, \rho)$ be a metric space of diameter at most $T/2$ and let
$\mathcal{P}(\mathbb{X})$ be the space of probability measures on $\mathbb{X}$,
such that, for any measure $\mu \in \mathcal{P}(\mathbb{X})$, its support
$\mathbb{X}_\mu$ is a compact set.
The space $\mathbb{X}_\mu$ is a natural object of interest in computational
topology. Its
persistent homology is usually approximated by the persistent homology
of the distance function to a
sample $X_N=\{x_1, \dots, x_N \} \subset \mathbb{X}_\mu$.
\cite{Fasy2013statistical} propose several methods for the construction of
confidence sets for the persistence diagram of $\mathbb{X}_\mu$, while
\cite{chazal2013optimal} establish optimal convergence rates for $d_b(
D_{\mathbb{X}_\mu}, D_{X_N} )$.

When $N$ is too large, the computation of the persistent homology of $X_N$ is
prohibitive, due to the combinatorial complexity of the computation. Our aim
is to study topological signatures of the data that can be efficiently computed
in a reasonable time. We define such quantities by repeatedly sampling $m$
points of $\mathbb{X}$ according to $\mu$.

For any positive integer $m$, let  $X = \{ x_1,\cdots , x_m \}
\subset \mathbb{X}$ be a sample of $m$ points from the measure $\mu \in
\mathcal{P}(\mathbb{X})$.
The corresponding persistence landscape is $\lambda_X$
and we denote by $\Psi_\mu^{m}$ the measure induced by $\mu^{\otimes
m}$ on $\mathcal{L_T}$.
Note that the persistence landscape $\lambda_X$  can
be seen as a single draw from the measure $\Psi_\mu^{m}$.
We consider the point-wise expectations of the (random) persistence landscape
under this
measure:
$\mathbb{E}_{\Psi_\mu^{m}}[\lambda_{X}(t)], t \in [0,T]$.
This quantity is relevant from a topological point of view, because
it is stable under perturbation of the underlying
measure $\mu$.
This stability result is the main result of the paper, presented
in detail in the next section.

The average landscape $\mathbb{E}_{\Psi_\mu^{m}}[\lambda_{X}]$ has a natural
empirical counterpart, which can be used as
its unbiased estimator. Let $S_1^m, \dots, S_n^m$ be $n$ independent samples of
size $m$ from $\mu$. We define the empirical average landscape as
\begin{equation}
\label{eq:avLand}
\overline{\lambda_n^m}(t)= \frac{1}{n} \sum_{i=1}^n \lambda_{S_i^m}(t), \quad
\text{ for all } t \in [0,T],
\end{equation}
and propose to  use
$\overline{\lambda_n^m}$ to estimate $\lambda_{{\mathbb{X}_\mu}}$.
The variance of this estimator under the $\ell_\infty$-distance was studied in
detail in \cite{chazal2013stochastic}. Here instead we are concerned with the
quantity  $\Vert
\lambda_{\mathbb{X}_\mu} -
\mathbb{E}_{\Psi_\mu^{m}}[\lambda_{X}] \Vert_{\infty}$, which can be seen as the
bias component (see Section \ref{sec:risk}).

In addition to the average,
we also consider using the \textit{closest sample} to
$\mathbb{X}_\mu$ in Hausdorff distance.
The closest sample method consists in choosing a sample of $m$ points of
$\mathbb{X}$, as close as possible to $\mathbb{X}_\mu$, and then use this sample
to build a landscape that approximates $\lambda_{{\mathbb{X}_\mu}}$.
Let $S_1^m, \dots, S_n^m$ be $n$ independent samples of size $m$ from
$\mu^{\otimes
m}$.
The closest sample is
\begin{equation}
\label{eq:closestSample}
\hat{C^m_n}= \arg\min_{S \in \{S_1^m, \dots, S_n^m\}} H(S, X_\mu)
\end{equation}
and the corresponding landscape function is
$\hat{\lambda_n^m} = \lambda_{{\hat{C^m_n}}}. $
Of course, the method requires the support of $\mu$ to be a known quantity.

\begin{remark}
Computing the persistent
homology of $X_N$ is $O(\exp(N))$, whereas computing
the average landscape is $O(n \exp(m))$ and the
persistent homology of the closest sample is $O(n m N + \exp(m) )$.
\end{remark}

\begin{remark}
The general framework described above is valid for the case in which
$\mu$ is a discrete measure with support $\mathbb{X}_\mu=\{x_1, \dots, x_N \}
\subset \mathbb{R}^D $.
For example, the following situation is very common in practice.  Let
$X_N=\{x_1, \dots, x_N\}$ be a given point cloud, for large but fixed $N \in
\mathbb{R}$. When $N$ is large, the computation of the persistent homology of
$X_N$ is unfeasible. Instead, we consider the discrete uniform measure $\mu$ that puts mass $1/N$ on each point of
$X_N$, and we propose to
estimate $\lambda_{{X_N}}=\lambda_{\mathbb{X}_u}$ by repeatedly subsampling $m \ll N$ points of $X_N$
according to $\mu$.
\end{remark}

We will study the
$\ell_\infty$-risk of the proposed estimators, $\mathbb{E}\left[ \Vert
\lambda_{\mathbb{X}_\mu} - \overline{\lambda_n^m} \Vert_{\infty} \right]$ and
$\mathbb{E}\left[ \Vert \lambda_{\mathbb{X}_\mu} - \hat{\lambda_n^m}
\Vert_{\infty} \right]$, under the following assumption on the underlying
measure $\mu$, which we will refer to as the {\it  $(a,b, r_0)$-standard
assumption}: there exist positive constants $a$, $b$ and $r_0 \geq 0$ such
that
\begin{equation}
\label{eq:ab}
\forall r>r_0 , \, \forall  x \in \mathbb{X}_\mu\ ,  \  \mu(B(x,r)) \geq 1\wedge
 a  r^b .
\end{equation}
For $r_0=0$, this is known as the $(a,b)$-standard assumption and has been
widely used in the literature of set estimation under Hausdorff distance
\citep{cuevas2004boundary,Cuevas09, singh2009adaptive} and more recently in the
statistical analysis of persistence diagrams \citep{chazal2013optimal,
Fasy2013statistical}.
We use the generalized version with $r_0 > 0$ to take into account the case in
which $\mu$ is a discrete measure (in which case $r_0$ depends on $N$); see Appendix \ref{sec:abtrunc} for more
details.

\section{Stability of the Average Landscape}
\label{sec:average}

Consider the framework described in Section \ref{sec:subsampling}: $m$ points
are repeatedly sampled from the space~$\mathbb{X}$ according to a measure $\mu
\in \mathcal{P}(\mathbb{X})$.
In this section, we show that the average landscape
$\mathbb{E}_{\Psi_\mu^{m}}[\lambda_{X}]$ is an interesting quantity on its own,
since it carries some stable topological information about the underlying
measure $\mu$, from which the data are generated.

\cite{chazal2013stochastic} provide a way to construct confidence bands for
$\mathbb{E}_{\Psi_\mu^m}[\lambda_X]$. Here, 
we compare the average landscapes corresponding to
two measures that are close to each other in the Wasserstein
metric.
\begin{definition}
Given a metric space $(\mathbb{X}, \rho)$, the $p$th Wasserstein distance
between two measures $\mu,\nu \in \mathcal{P}(\mathbb{X})$ is
$W_{\rho,p}(\mu,\nu)= \left( \inf_\Pi \int_{\mathbb{X}\times\mathbb{X}}
[\rho(x,y)]^p d\Pi(x,y)  \right)^{\frac{1}{p}}$,
where the infimum is taken over all measures on $\mathbb{X}\times\mathbb{X}$
with marginals $\mu$ and $\nu$.
\end{definition}
The Wasserstein distance is often colloquially referred to as the earth-movers
distance, as $\Pi$ can be seen as a transport plan.
The following result shows that the average behavior of the landscapes of sets
of $m$ points sampled according to any measure $\mu$ is stable with respect to
the Wasserstein distance. 
\begin{theorem} \label{thm:stabW}
Let $(\mathbb{X}, \rho)$ be a metric space of diameter bounded by $T/2$.
Let $X \sim \mu^{\otimes m}$ and $Y \sim \nu^{\otimes m}$, where $\mu, \nu \in \mathcal{P}(\mathbb{X})$ are two probability measures.
For any $p \geq 1$ we have
\begin{equation*}
\left\Vert \mathbb{E}_{\Psi_\mu^m}[\lambda_X] - \mathbb{E}_{\Psi_\nu^m}[\lambda_Y]
\right\Vert_\infty \leq  m^{\frac{1}{p}} W_{\rho,p}(\mu,\nu).
\end{equation*}
%where $\Vert f \Vert_\infty = \sup_{t \in [0,T]} |f(t)|$ is the supremum norm.
\end{theorem}

For measures that are not defined on the same metric space, the inequality of
Theorem \ref{thm:stabW} can be extended to Gromov-Wasserstein metric:
$\displaystyle
\left\Vert \mathbb{E}_{\Psi_\mu^m}[\lambda_X] - \mathbb{E}_{\Psi_\nu^m}[\lambda_Y]
\right\Vert_\infty \leq  2 m^{\frac{1}{p}} GW_{\rho,p}(\mu,\nu).
$

The result of Theorem \ref{thm:stabW} is useful for two reasons. First, it tells
us that for a fixed $m$, the expected ``topological behavior'' of a set of $m$
points carries some stable information about the underlying measure from which
the data are generated. Second, it provides a lower bound for the Wasserstein
distance between two measures, based on the topological signature of samples of
$m$ points.
%In practice, when we can only access a measure $\eta$ through a large data set of $N$
%points, we usually approximate $\eta$ with its empirical version $\eta_N$, the
%discrete measure that put mass $1/N$ on each point of the data set.
%Estimating $\mathbb{E}_{\Phi_\eta^{m}}(\lambda)$ with
%$\mathbb{E}_{\Phi_{\eta_N}^{m}}(\lambda)$ gives rise to a bias controlled by
%$W_p(\eta, \eta_N)$.

The dependence on $m$ of the upper bound of Theorem \ref{thm:stabW} seems to be
necessary in this setting: intuitively, when $m$ grows, the samples of $m$
points converge to the support of $\mu$ and $\nu$ w.r.t. the Hausdorff distance
respectively. Therefore the expected landscapes should converge to the
landscapes of the support of the measures. But, in general, two measures that
are close in the Wasserstein metric can have support that have very different
and unrelated topologies.
Indeed, a similar dependence was also obtained in
\cite{blumberg2012persistent} when analyzing the stability properties of
persistent diagrams in the Gromov-Prohorov metric.

Note that in Theorem \ref{thm:stabW} we do not make any assumption on the
measures $\mu$ and $\nu$. If we assume that they both satisfy the $(a,b,r_0)$-standard
assumption we can provide a different  bound on the difference of the expected
landscapes, based on the Hausdorff distance between the support of the two
measures.

\begin{theorem} \label{thm:stabH}
Let $(\mathbb{X}, \rho)$ be a metric space of diameter bounded by $T/2$.
Let $X \sim \mu^{\otimes m}$ and $Y \sim \nu^{\otimes m}$, where
$\mu, \nu \in \mathcal{P}(\mathbb{X})$ satisfy the $(a,b,r_0)$-standard assumption on $\mathbb{X}$. 
Define $r_m=2\left(\frac{\log m}{am} \right)^{1/b}$. Then
$$
\Vert \mathbb{E}_{\Psi_\mu^{m}}(\lambda_X) - \mathbb{E}_{\Psi_\nu^{m}}(\lambda_Y)
\Vert_\infty \leq
H(\mathbb{X}_\mu, \mathbb{X}_\nu) + 2r_0 \, + \, 2r_m \mathbbm{1}_{(r_0,\infty)}( r_m) \, + \,2 \, C_1(a,b) \, r_m \, \frac{1}{(\log
m)^2},
$$
where $C_1(a,b)$ is a constant depending on $a$ and $b$.
\end{theorem}

The following result follows by combining theorems \ref{thm:stabW} and
\ref{thm:stabH}.
\begin{corollary}
Under the same assumptions of Theorem \ref{thm:stabH} we have that
\begin{align*}
\left\Vert \mathbb{E}_{\Psi_\mu^{m} }(\lambda_X) - \mathbb{E}_{\Psi_\nu^{m}
}(\lambda_Y) \right\Vert_\infty  \leq
\min &
\Big\{  m^{\frac{1}{p}} W_p(\mu,\nu),  \\
H(\mathbb{X}_\mu, \mathbb{X}_\nu) + &2r_0 \, + \, 2r_m \mathbbm{1}_{(r_0,\infty)}( r_m) \, + \,2 \, C_1(a,b) \, r_m \, \frac{1}{(\log
m)^2}
\Big\}.
\end{align*}

\end{corollary}

%\todo{Discussion: This allow us to give a bound on the difference between
%$\mathbb{E}_{\Psi_\mu^{m}}(\lambda)$ and
%$\mathbb{E}_{\Psi_{\mu_N}^{m}}(\lambda')$.
%Thanks to the paper submitted to SoCG 2014, $\mathbb{E}_{\Phi_\mu^{m}}(\lambda)$
%gives us a family of robust topological objects associated to a measure that are
%easy to approximate (even for huge data sets!) $\to$ confidence bands,
%statistical tests. This could lead to a very useful tools in some applications.
%}

%\textbf{Remark. The case of \cite{chazal2013optimal}} \todo{not sure where to
%put this remark}
%Because of the deviation inequality proven in \cite{chazal2013optimal}, it
%appears that averaging the landscape does not really improve with respect to the
%strategy consisting in considering only one random sample.
%
%Assume that $\mu$ is a probability measure satisfying a $(a,b)$-standard
%assumption and that $\mathbb{X}$ is endowed with the probability $\mu^{\otimes
%m}$. One easily deduces from the deviation inequality in
%\cite{chazal2013optimal}  that $\mathbb{E}_{\Psi_\mu^m}(\lambda)$ converges to
%$\lambda_{{X_\mu}}$, where $X_\mu$ is the support of $\mu$, when $m \to
%+\infty$ with a convergence rate $O(\frac{\log m}{m^{1/b}})$ (we can give an
%explicit bound). However, because of the deviation inequality, we already know
%that $\lambda_{{X_m}}$ also converges to $\lambda_{{X_\mu}}$ with the same
%convergence rate.
%

%----------------------------------------------------------
\section{Risk Analysis}
%----------------------------------------------------------
\label{sec:risk}
In this section we study the performance of
the average landscape $\overline{\lambda_n^m}$ and of the landscape of the closest sample $\hat{\lambda_n^m}$, as estimators of $\lambda_{{\mathbb{X}_\mu}}$.
We start by decomposing the $\ell_\infty$-risk of the average landscape as
follows. Set $\lambda_1 = \lambda_{S_1^m}$, with $S_1^m$ a sample of size $m$ from
$\mu$. Then, 
\begin{equation}
\E    \left\Vert \lambda_{\mathbb{X}_\mu}   -  \overline{\lambda_n^m}
\right\Vert_{\infty}  \leq  \left\Vert \lambda_{\mathbb{X}_\mu}  - \E  \lambda_1
 \right\Vert_{\infty} +  \E   \left\Vert \overline{\lambda_n^m} - \E \lambda_1
\right\Vert_{\infty},
\label{eq:risk}
\end{equation}
where the expectation of $\overline{\lambda_n^m}$ is wrt
$(\Psi_{\mu}^m)^{\otimes n}$ and the expectation of $\lambda_1$ is wrt $\Psi^m_{\mu}$.

For the bias term $\left\Vert \lambda_{\mathbb{X}_\mu} - \E  \lambda_1
\right\Vert_{\infty} $ we use the stability property to go back into $\R^d$ :
\begin{equation}
\left\Vert \lambda_{\mathbb{X}_\mu}  - \E  \lambda_1  \right\Vert_{\infty} 
\leq  \E_{\Psi_{\mu}^m}   \left\Vert \lambda_{\mathbb{X}_\mu}  -  \lambda_1
\right\Vert_{\infty}  \leq
\E_{\mu^{\otimes m} }  H( \mathbb{X}_\mu  , X ),
\label{eq:hausdorff}
\end{equation}
where $X$ is a sample of size $m$ from $\mu$.
Note that, if calculating $H( \mathbb{X}_\mu  , X)$ is computationally feasible,
then, in practice, $\E_{\mu^{\otimes m} }  H( \mathbb{X}_\mu  , X )$ can be
approximated by the average of a large number $B$ of values of $H (
\mathbb{X}_\mu  , X)$, for $B$ different draws of subsamples $X \sim
\mu^{\otimes m}$. 
%See Figure \ref{fig:exampleBias}.
% \begin{figure}[!ht]
% \centering
% \includegraphics[height=2.7in]{figs/exBiasCircle}
% \caption{Pink: 95\% uniform confidence band for $\mathbb{E}\lambda_1$ around
% $\overline{\lambda_n^m}$, constructed using the multiplier bootstrap described
% in \cite{chazal2013stochastic}. Green: bound on the bias term computed as
% $B^{-1} \sum_{i=1}^B H( \mathbb{X}_\mu  , \{Z_1,\ldots,Z_m \}_i)$, for $B=100$.
% }
% \label{fig:exampleBias}
% \end{figure}

To give an explicit bound on the bias we assume that $\mu$ satisfies the
$(a,b,r_0)$-standard assumption. 

\begin{theorem} 
\label{th:riskAverage}
Let $r_m=2\left(\frac{\log m}{am} \right)^{1/b}$. If $\mu$ satisfies the $(a,b,r_0)$-standard assumption, then
$$
\left\Vert \lambda_{\mathbb{X}_\mu}  - \E  \lambda_1  \right\Vert_{\infty} \leq r_0 \, + \,  r_m \, \mathbbm{1}_{(r_0,\infty)}( r_m) \, + \, C_1(a,b)  \, r_m \, \frac{1}{(\log m)^2},
$$
where $C_1(a,b)$ is a constant that depends on $a$ and $b$.
\end{theorem}
\cite{chazal2013stochastic} controls the variance term, which is of the order of
$1/ {\sqrt n}$. Therefore, if $r_0$ is negligible,  we see that $n$ should be taken of the order of $\left(m/{\log m} \right) ^{2 / b}$.

We now turn to the closest sample estimator $\hat{\lambda}_n$ and investigate
its $\ell_\infty$ risk $
\mathbb{E}\left[ \Vert \lambda_{\mathbb{X}_\mu} - \hat{\lambda_n^m} \Vert_{\infty} \right]$, 
where the expectation is with respect to $( \Psi^m_{\mu})^{\otimes
n}$.
As before, in our analysis we rely on the stability property $ \mathbb{E}\left[ \Vert
    \lambda_{\mathbb{X}_\mu} - \hat{\lambda_n^m} \Vert_{\infty} \right] \leq
    \mathbb{E} \left[    H(\mathbb{X}_\mu, \widehat{C_n^m} ) \right]$, where the
    second expectation is with respect to $( \mu^{\otimes m})^{\otimes n}$.

\begin{theorem} \label{theo:RiskClosest}
Let $r_m=2\Big( \frac{\log(2^b m)}{a m}\Big)^{\frac{1}{b}}$.  If
$\mu \in \mathcal{P}(\mathbb{X})$ satisfies the $(a,b,r_0)$-standard
assumption, then
$$ 
\mathbb{E}\left[ \Vert \lambda_{\mathbb{X}_\mu} - \hat{\lambda_n^m}
\Vert_{\infty} \right] \leq r_0 \, + \, r_m
\mathbbm{1}_{(r_0,\infty)} (r_m) \, + \, C_2(a,b) \, r_m \, \frac{1}{n \; [ \, \log(2^b m)]^{n+1}},
$$
where $C_2(a,b)$ is a constant that depends on $a$ and $b$.
\end{theorem}

\begin{remark}
The risk of the closest
subsample method can in principle
be smaller than the average landscape method.
In Appendix \ref{sec:abtrunc} we show that if $\mu$ is the discrete uniform measure on a point cloud of size $N$, sampled from a measure satisfying the $(a,b,0)$-standard assumption, then $r_0$ is of the order of $(\frac{\log N}{N})^{1/b}$. 
When $r_0$ is negligible, the rates of theorems \ref{th:riskAverage} and \ref{theo:RiskClosest} are comparable, both of the order of $O(\frac{\log m}{m})^{1/b}$.
However, the average method
has another advantage: it is robust to outliers.
This point is discussed in detail in Appendix \ref{sec:robustness}.
\end{remark}

%----------------------------------------------------------
\section{Experiments}
%----------------------------------------------------------
\label{sec:experiments}

 \vspace{-2pt}

In this section, we illustrate our methods. Since computing the
persistent homology of the Vietoris-Rips (VR) filtrations built on top
of the large samples is infeasible, we resort to the subsampling
strategy described in Section \ref{sec:subsampling}. More formally,
let $X_N=\{x_1, \dots, x_N \}$ and $Y_N=\{y_1, \dots, y_N \}$ be two
large point clouds. We draw $n$ subsamples
each of size $m \ll N$ points from
$\mu$ and $\nu$, the discrete uniform measures on $X_N$ and $Y_N$, and
we compare the corresponding average landscapes and closest subsample
landscapes, induced by the persistent homology of the VR filtrations
built on top of the subsamples.  We apply this technique to two
examples.
%The R code is available on the website of the CMU
%TopStat Group\footnote{http://www.stat.cmu.edu/topstat/}
%and requires the R package \textsf{\textbf{persistence}}.

\begin{figure}[!ht]
\centering
\includegraphics[width=5.3in]{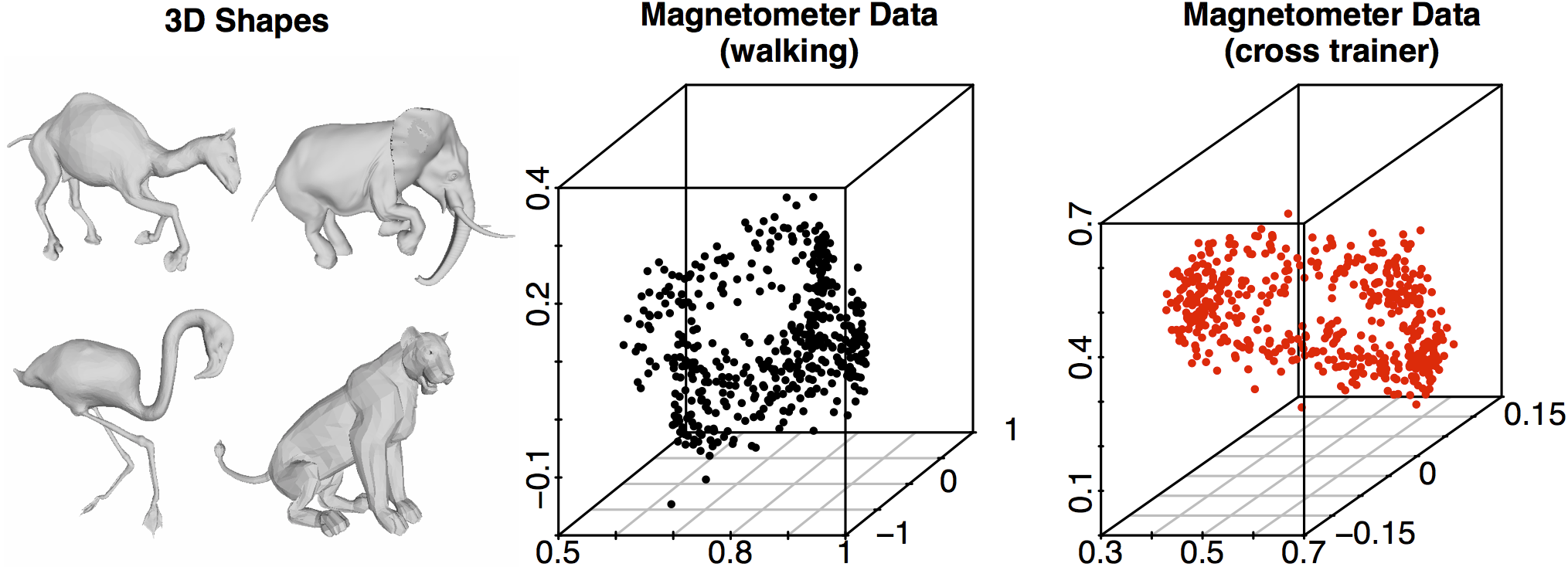}
\caption{Left: 3D shapes of the first experiment.
Middle and Left: 500 random points from the magnetometer data of the second
experiment.}
\label{fig:data}
\end{figure}

{\bf 3D Shapes.}  We use the publicly available database of
triangulated shapes \citep{sumner2004deformation}. We select a
single {\it pose} (\#2) of 4 different classes: {\it camel, elephant,
flamingo, lion}. The 4 shapes are represented in Figure
\ref{fig:data}. In practice, each shape consists of a 3D point
cloud embedded in the Euclidean space, with a number of vertices that
ranges from 7K to 40K.  The data are normalized, so that the diameter
of each shape is 1.  For $n=100$ times we subsample $m=300$ points
from each shape; then we select the closest subsample to the
corresponding original point cloud and compute $4\times n$ persistence
diagrams (dimension 1), one for each subsample. See Figure \ref{fig:animals}:
the
plot on the left shows the landscapes corresponding to the closest
subsamples of $m$ points among the $n$ different subsamples from each
shape; the plot in the middle shows the empirical average landscapes
within each class, computed as the pointwise average of $n$
landscapes, with a 95\% uniform confidence band for the true average
landscape, constructed using the method described in
\cite{chazal2013stochastic}; the dissimilarity matrix on the right
shows the pairwise $\ell_\infty$ distances between the average
landscapes (scale from yellow to red), which, according to Theorem
\ref{thm:stabW}, represent a lower bound for the pairwise Wasserstein
distances of the discrete uniform measures on the 4 different shapes.

\begin{figure}[!ht]
\centering
\includegraphics[width=5.4in]{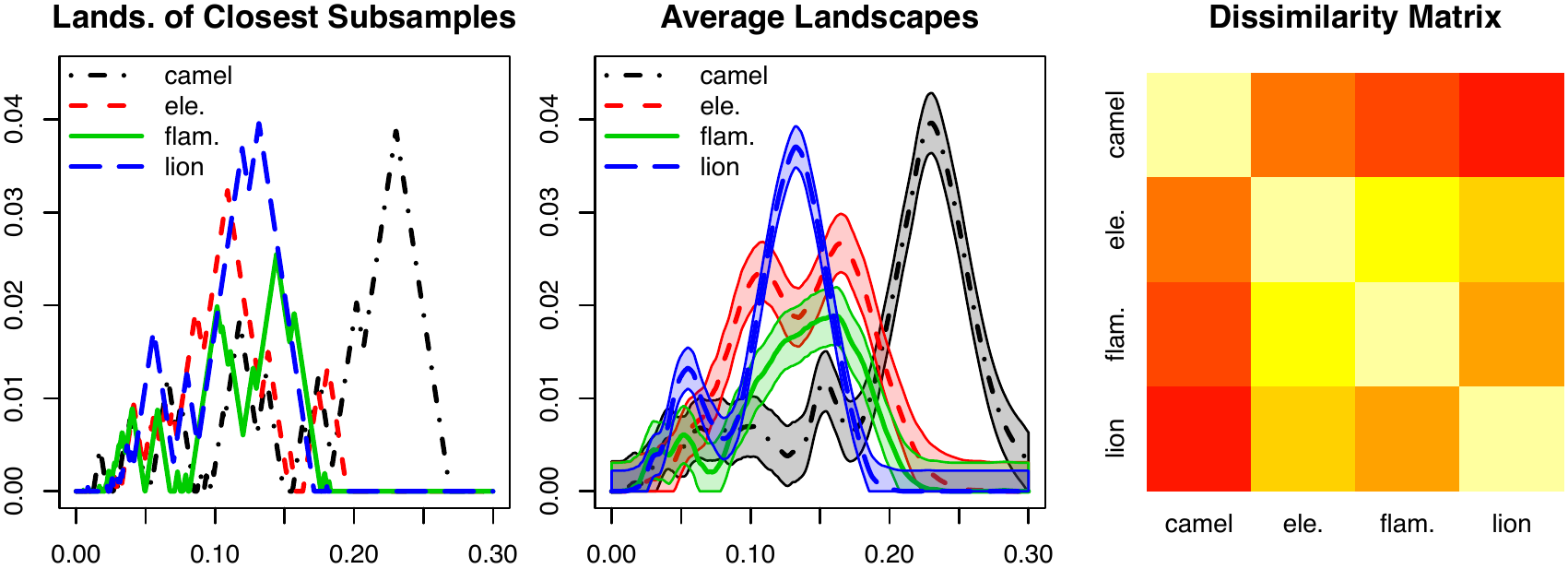}
\caption{Subsampling methods applied to 3D shapes. For $n=100$ subsamples of
size $m=300$, for each shape, we constructed the landscapes of the closest
subsample (left), the average landscape with 95\% confidence band (middle) and
the dissimilarity matrix of the pairwise $\ell_\infty$ distance between average
landscapes.}
\label{fig:animals}
\end{figure}

\begin{figure}[!ht]
\centering
\includegraphics[width=5.4in]{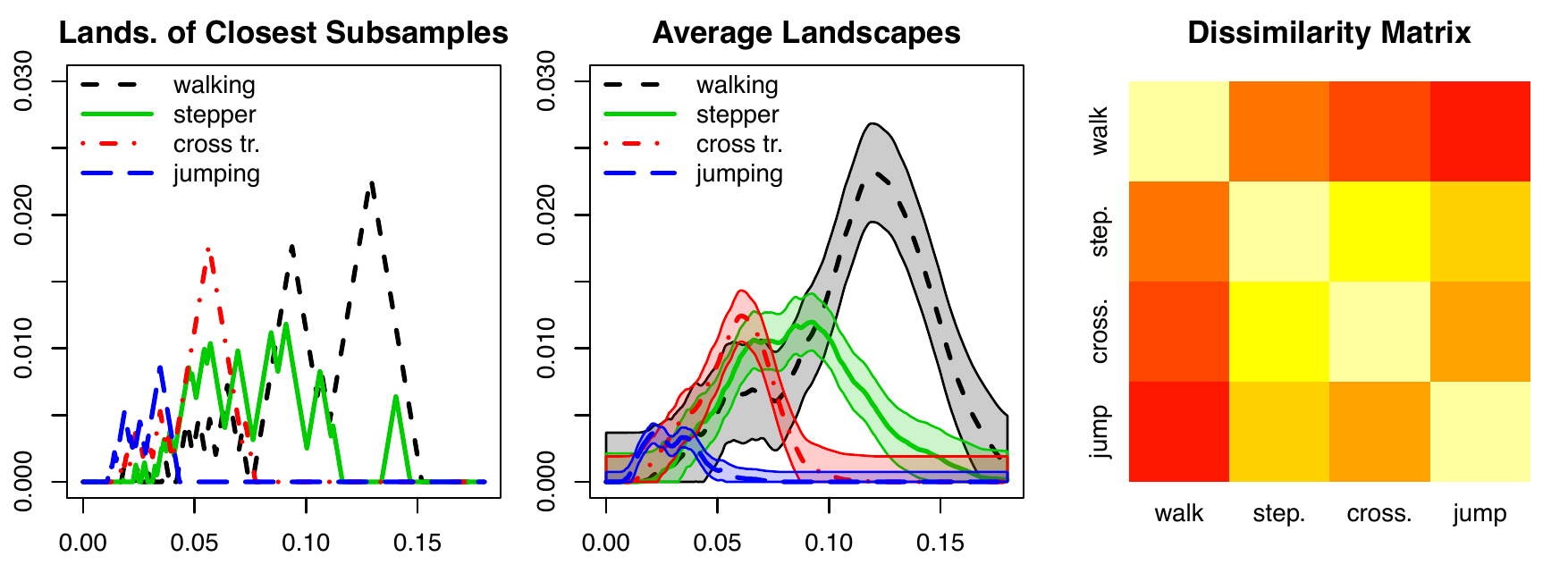}
\caption{Subsampling methods applied to magnetometer data.  For $n=80$
subsamples of size $m=200$, for each activity, we constructed the landscapes of
the closest subsample (left), the average landscape with 95\% confidence band
(middle) and the dissimilarity matrix of the pairwise $\ell_\infty$ distance
between average landscapes.}
\label{fig:activities}
% \vspace{-3pt}
\end{figure}
{\bf Magnetometer Data.}
For the second example, we consider the problem of distinguishing  human
activities performed while wearing inertial and magnetic sensor units. The
dataset is publicly available at the UCI Machine Learning
Repository\footnote{
http://archive.ics.uci.edu/ml/datasets/Daily+and+Sports+Activities} and is
described in \cite{barshan2013recognizing}, where it is used to classify 19
activities performed by 8 people that wear sensor units on the chest, arms and
legs. For ease of illustration, we report here the results on 4 activities
(walking, stepper, cross trainer, jumping) performed by a single person (\#1).
We use the data from the magnetomer of a single sensor (left leg), which
measures the direction of the magnetic field in the space at a frequency of
25Hz. For each activity there are 7,500 consecutive measurements that we treat
as a 3D point cloud in the Euclidean space. As an example, Figure \ref{fig:data}
shows 500 points at random for 2 activities (walking and using a cross trainer).
As in the previous example, for $n=80$ times, we subsample $m=200$ points  from
the point cloud of each activity, construct the landscapes of the closest
subsamples, the average landscapes (dim 1), and the dissimilarity matrix based
on the $\ell_\infty$ distances of the average landscapes. See Figure
\ref{fig:activities}.
To the best of our knowledge persistent homology has never been used to
study data from  accelerometers or magnetometers before. A remarkable advantage
is that the methods of persistent homology are insensitive to the orientation of
the input data, as opposed to other methods that require the exact calibration
of the sensor units; see, for example, \cite{altun2010comparative} and
\cite{barshan2013recognizing}.

\section{Conclusion}
\label{sec:conclusion}

We have presented a framework for
approximating the persistent homology
of a set using subsamples.
The method is simple and computationally fast.
Moreover, we provided stability results for the new summaries and bounds on the 
risk of the proposed estimators.
In the future we will release software for implementing
the method.
We plan to investigate other methods for
speeding up the computations.

% TO ADD FOR FINAL VERSION ONLY

%\newpage

%\subsubsection*{References}
{\small
\bibliography{ref}
}

\newpage

\appendix

\begin{center}
{\Large \bf
Appendix
}
\end{center}

\section{Technical results}
\label{sec:technical}

In this section we present some technical results that will be used to prove 
the main theorems. \\
First, we expand the notation introduced in the body of the paper (Section 3).
For any positive integer $m$, let $\phi_m : \mathbb{X}^m \to \mathcal{D}_T$ 
be the diagram
and
$\psi_m : \mathcal{D}_T \to \mathcal{L}_T$ the landscape, i.e.
$\phi_m(X) = D_{X} , \ \ \mbox{\rm for any } \ \ X = \{ x_1,\cdots , x_m \}
\subset \mathbb{X}$
and
$\psi(D_X) = \lambda_{D_X}=\lambda_X , \ \ \mbox{\rm for any } \ \ D_X \in \mathcal{D}_T.$
Given $\mu \in \mathcal{P}(\mathbb{X})$, we denote by $\Phi_\mu^{m}$ the
push-forward measure of $\mu^{\otimes m}$ by $\phi_m$, that is $\Phi_\mu^{m} =
(\phi_m)_*\mu$.
Similarly, we denote by $\Psi_\mu^{m}$ the push-forward (induced) measure of $\mu^{\otimes
m}$ by $\psi \circ \phi_m $, that is $\Psi_\mu^{m} = (\psi \circ \phi_m)_*\mu$.

For a fixed integer $m>0$, consider the metric space $(\mathbb{X}, \rho)$ and
the space $\mathbb{X}^m$ endowed with a metric~$\rho_m$. We impose two
conditions on $\rho_m$:
\begin{enumerate}
\item[(C1)] Given a real number $p\geq 1$, for any $X=\{x_1, \dots, x_m\}
\subset \mathbb{X}$ and $Y=\{y_1, \dots, y_m\} \subset \mathbb{X}$,
\begin{equation}
\rho_m(X,Y) \leq \left(\sum_{i=1}^m \rho(x_i,y_i)^p \right)^\frac{1}{p},
\label{eq:dist1}
\end{equation}
\item[(C2)] For any $X=\{x_1, \dots, x_m\} \subset \mathbb{X}$ and $Y=\{y_1,
\dots, y_m\} \subset \mathbb{X}$,
\begin{equation}
H(X,Y) \leq \rho_m(X,Y).
\label{eq:dist2}
\end{equation}
\end{enumerate}

Two examples of distance that satisfy conditions (C1) and (C2) are the Hausdorff
distance and the $L_p$-distance
$
\rho_m(X,Y)= \left(\sum_{i=1}^m \rho(x_i,y_i)^p \right)^\frac{1}{p}.
$

\begin{lemma}
\label{lemma:ineq1}
For any probability measures $\mu, \nu \in \mathcal{P}(\mathbb{X})$ and any
metric $\rho_m: \mathbb{X}^m\times\mathbb{X}^m \to \mathbb{R}$ that satisfies
(C1), we have
$$
W_{\rho_m,p}(\mu^{\otimes m}, \nu^{\otimes m}) \leq m^{\frac{1}{p}}
W_{\rho,p}(\mu,\nu) \/.$$
\end{lemma}
\noindent
{\bf Remark:} The bound of the above lemma is tight: it is an equality when
$\mu$ is a Dirac measure and $\nu$ any other measure.
\begin{proof}
Let $\Pi \in \mathcal{P}(\mathbb{X} \times \mathbb{X})$ be a transport plan
between $\mu$ and $\nu$. Up to reordering the components of $\mathbb{X}^{2m}$,
$\Pi^{\otimes m}$ is a transport plan between $\mu^{\otimes m}$ and
$\nu^{\otimes m}$ whose $p$-cost is given by
\begin{eqnarray*}
%{C}_p(\Pi^{\otimes m})^p & = &
\int_{\mathbb{X}^m \times \mathbb{X}^m}
\rho_m(X,Y)^p d\Pi^{\otimes m}(X,Y)
 & \leq & \int_{\mathbb{X}^m \times \mathbb{X}^m} \sum_{i=1}^m \rho(x_i,y_i)^p
\; d\Pi(x_1,y_1) \cdots d\Pi(x_m,y_m) \\
 & = & m \int_{\mathbb{X} \times \mathbb{X}} \rho(x_1,y_1)^p d\Pi(x_1,y_1).
 %= m C_p(\Pi)^p
\end{eqnarray*}
The lemma follows by taking the minimum over all transport plans on
both sides of this inequality.
\end{proof}

\begin{lemma}
\label{lemma:ineq2}
For any probability measures $\mu, \nu \in \mathcal{P}(\mathbb{X})$ and any
metric $\rho_m: \mathbb{X}^m\times\mathbb{X}^m \to \mathbb{R}$ that satisfies
(C2), we have
$$
W_{d_b,p}\left( \Phi_\mu^{m} , \;  \Phi_\nu^{m} \right) \leq
W_{\rho_m,p}(\mu^{\otimes m}, \nu^{\otimes m}).
$$
\end{lemma}
\begin{proof}
This is a consequence of the stability theorem for persistence
diagrams.
Given $X,Y \subset \mathbb{X}^m$, define
$$
\Lambda_m (X,Y)= \left( D_X, D_Y \right).
$$
If $\Pi \in \mathcal{P}(\mathbb{X}^m \times \mathbb{X}^m)$ is a transport plan
between $\mu^{\otimes m}$ and $\nu^{\otimes m}$ then $\Lambda_{m,*}\Pi$ is a
transport plan between $\Phi_\mu^{m}$ and $\Phi_\nu^{m}$. Its $p$-cost is given
by
\begin{eqnarray*}
%C_p(\Lambda_{m,*}\Pi)^p & = &
\int_{\mathcal{D}_T \times \mathcal{D}_T} d_b(D_X,
D_Y)^p d\Lambda_{m,*}\Pi(D_X,D_Y)
 & = & \int_{\mathbb{X}^m \times \mathbb{X}^m} d_b(\phi_m(X),\phi_m(Y))^p
d\Pi(X,Y) \\
  & \leq & \int_{\mathbb{X}^m \times \mathbb{X}^m} H(X,Y)^p d\Pi(X,Y)\ \
(\mbox{\rm stability theorem })\\
  & \leq & \int_{\mathbb{X}^m \times \mathbb{X}^m} \rho_m(X,Y)^p d\Pi(X,Y).
\end{eqnarray*}
The lemma follows by taking the minimum over all transport plans on
both sides of this inequality.
\end{proof}

\begin{lemma}
\label{lemma:ineq3}
Let $\mu$ and $\nu$ be two probability measures  on $\mathbb{X}$. Let $\lambda_X
\sim \Psi_\mu^m$ and $\lambda_Y \sim \Psi_\nu^m$. Then
$$
\left\Vert \mathbb{E}_{\Psi_\mu^m}[\lambda_X] - \mathbb{E}_{\Psi_\nu^m}[\lambda_Y]
\right\Vert_\infty \leq W_{d_b,p}\left( \Phi_\mu^{m} , \;  \Phi_\nu^{m} \right).
$$
\end{lemma}
\begin{proof}
Let $\Pi$ be a transport plan between $\Phi_\mu^{m}$ and $\Phi_\nu^{m}$. For any
$t \in \mathbb{R}$ we have
\begin{eqnarray*}
\left |  \mathbb{E}_{\Psi_\mu^m}[\lambda_X](t) -
\mathbb{E}_{\Psi_\nu^m}[\lambda_Y](t) \right |^p
& = & \left |  \mathbb{E}[\lambda_X(t) - \lambda_Y(t)] \right |^p \\
& \leq &   \mathbb{E}\left[ | \lambda_X(t) - \lambda_Y(t) |^p \right]  \ \
(\mbox{\rm Jensen inequality}) \\
& \leq &   \mathbb{E}\left[ d_b(D_X, D_Y)^p \right]    \ \ (\mbox{\rm
Stability of landscapes}) \\
& = &  \int_{\mathcal{D}_T \times \mathcal{D}_T} d_b(D_X,D_Y)^p d\Pi(D_X, D_Y) \\
 & = & C_p(\Pi)^p.
\end{eqnarray*}
\end{proof}

The following lemma is similar to Theorem 2 in \cite{chazal2013optimal}.
\begin{lemma}
\label{lem:covering}
Let $X$ be a sample of size $m$ from a measure $\mu \in \mathcal{P}(\mathbb{X})$
that satisfies the $(a,b,r_0)$-standard assumption.
Let $r_m=2\left(\frac{\log m}{am} \right)^{1/b}$. Then
$$
\mathbb{E}\left[ H( X, \mathbb{X}_\mu )  \right] \leq r_0 \, + \,
2\left(\frac{\log m}{am} \right)^{1/b} \mathbbm{1}_{(r_0,\infty)}( r_m) \, + \, 2 \, C_1(a,b) \left(\frac{\log m}{am}
\right)^{1/b} \frac{1}{(\log m)^2},
$$
where $C_1(a,b)$ is a constant depending on $a$ and $b$.
\end{lemma}
\begin{proof}
Let $r>r_0$. It can be proven that
$
q:=\text{Cv} \left( \mathbb{X}_\mu, r/2 \right) \leq \frac{4^b}{a r^b}  \vee 1
$, where Cv($ \mathbb{X}_\mu,2 r$) denotes  the number of balls of radius $r/2$
that are necessary to cover $\mathbb{X}_\mu$ . Let $\mathcal{C}=\{x_1, \dots ,
x_{p}\}$  be a set of centers such that  $B(x_1, r/2), \dots , B(x_p, r/2)$ is a
covering of $\mathbb{X}_\mu$.
Then,
\begin{align*}
\mathbb{P} \left( H( X, \mathbb{X}_\mu )  > r \right)
& \leq   \mathbb{P}\left(H(X, \mathcal{C})+H(\mathcal{C}, \mathbb{X}_\mu) > r
\right)\\
& \leq   \mathbb{P}\left(H(X, \mathcal{C})  > r/2 \right)\\
& \leq  \mathbb{P}\left( \exists i \in \{ 1, \cdots, p  \} \ \mbox{\rm such
that} \  X   \cap B(x_i, r/2) = \emptyset \right)    \\
 & \leq  \sum_{ i = 1 } ^p   \mathbb{P} \left( X   \cap B(x_i, r/2) = \emptyset
\right)  \\
  & \leq    \frac{4^b}{a r^b}     \left[  1 -    \inf_{i=1 \dots p }
\mathbb{P}(B(x_i, r/2)) \right] ^m   \\
& \leq \frac{4^b}{a r^b}     \left[  1 -     \frac{a  r^b}{2^b} \right] ^m  \\
& \leq  \frac{4^b}{a r^b}    \exp\left( -   m   \frac{ a}{2^b}  r^b\right)
\end{align*}
Then
\begin{align}
\mathbb{E}\left[ H( X, \mathbb{X}_\mu )  \right] &= \int_{r>0} \mathbb{P} \left(
H( X, \mathbb{X}_\mu )  > r \right) dr \nonumber\\
& \leq r_0 + \int_{r>r_0} \mathbb{P} \left( H( X, \mathbb{X}_\mu )  > r \right)
dr . \label{eq:r0}
\end{align}
If $r_m \leq r_0$ then the
last quantity in \eqref{eq:r0} is bounded by
$$
r_0 + \int_{r>r_m} \mathbb{P} \left( H( X, \mathbb{X}_\mu )  > r \right) dr,
$$
otherwise \eqref{eq:r0} is bounded by
$$
r_0 + \int_{r>0} \mathbb{P} \left( H( X, \mathbb{X}_\mu )  > r \right) dr \leq
r_0 + r_m +\int_{r>r_m} \mathbb{P} \left( H( X, \mathbb{X}_\mu )  > r \right)
dr.
$$
In either case, we follow the strategy in \cite{chazal2013optimal} to obtain the
following bound:
\begin{align*}
\int_{r>r_m} \mathbb{P} \left( H( X, \mathbb{X}_\mu )  > r \right) dr \leq
2\, C(a,b) \left(\frac{\log m}{am} \right)^{1/b} \frac{1}{(\log m)^2},
\end{align*}
which implies that
$$
\mathbb{E}\left[ H( X, \mathbb{X}_\mu )  \right] \leq r_0 \, + \, r_m
\mathbbm{1}_{(r_0,\infty)}(r_m) \, + \, 2\, C(a,b) \left(\frac{\log m}{am} \right)^{1/b}
\frac{1}{(\log m)^2}.
$$
\end{proof}

\section{Main Proofs}
\label{sec:proofs}
\paragraph{Proof of Theorem \ref{thm:stabW}}
It immediately follows from the three following inequalities of
Lemmas \ref{lemma:ineq1}, \ref{lemma:ineq2} and \ref{lemma:ineq3}:
\begin{itemize}
\item upperbound on the Wasserstein distance between the tensor product of
measures:
$$
W_{\rho_m,p}(\mu^{\otimes m}, \nu^{\otimes m}) \leq m^{\frac{1}{p}}
W_{\rho,p}(\mu,\nu)
$$
\item from measures on $\mathbb{X}^m$ to measures on $\mathcal{D}$:
$$
W_{d_b,p}\left( \Phi_\mu^{m} , \;  \Phi_\nu^{m} \right) \leq
W_{\rho_m,p}(\mu^{\otimes m}, \nu^{\otimes m})
$$
\item from measures on $\mathcal{D}$ to difference of the expected landscapes:
$$
\left\Vert \mathbb{E}_{\Psi_\mu^m}[\lambda_X] - \mathbb{E}_{\Psi_\nu^m}[\lambda_Y]
\right\Vert_\infty \leq W_{d_b,p}\left( \Phi_\mu^{m} , \; \Phi_\nu^{m} \right)
$$
\end{itemize}
\qed

\paragraph{Proof of Theorem \ref{thm:stabH}}
\begin{align}
\| \mathbb{E}_{\Psi_\mu^{m}}(\lambda_X) - \mathbb{E}_{\Psi_\nu^{m}}(\lambda_Y)
\|_\infty &=
 \int_{\varepsilon>0} \mathbb{P}_{\Psi_\mu^{m} \otimes \Psi_\nu^{m}} \left(
\Vert \lambda_X - \lambda_Y \Vert_\infty > \varepsilon \right) d\varepsilon
\nonumber \\
 &= \varepsilon_0 +  \int_{\varepsilon>\varepsilon_0} \mathbb{P}_{\Psi_\mu^{m}
\otimes \Psi_\nu^{m}} \left( \Vert \lambda_X - \lambda_Y \Vert_\infty > \varepsilon
\right) d\varepsilon.
\label{eq:intLambda}
\end{align}
The event $\{ \Vert \lambda_X - \lambda_Y \Vert_\infty > \varepsilon \}$ inside the
integral implies that
\begin{equation} \label{eq:eps0}
\varepsilon_0 \leq \varepsilon < H (X,Y) \leq H (X,\mathbb{X}_\mu) + H
(\mathbb{X}_\mu,\mathbb{X}_\nu) + H (Y,\mathbb{X}_\nu),
\end{equation}
where X and Y are two samples of $m$ points from $\mu$ and $\nu$, respectively.
Let $\varepsilon_0= H(\mathbb{X}_\mu, \mathbb{X}_\nu)$. By \eqref{eq:eps0} it
follows that at least one of the following conditions holds:
\begin{align*}
& H (X, \mathbb{X}_\mu) \geq \frac{\varepsilon-\varepsilon_0}{2},\\
& H (Y,\mathbb{X}_\nu) \geq \frac{\varepsilon-\varepsilon_0}{2}.
\end{align*}
We assume that the first condition holds (the other case follows similarly).
Then the last quantity in equation \eqref{eq:intLambda} can be bounded by
\begin{align*}
& \varepsilon_0 +  \int_{\varepsilon>\varepsilon_0} \mathbb{P}\left(
H(X,\mathbb{X}_\mu) \geq \frac{\varepsilon-\varepsilon_0}{2} \right)
d\varepsilon \\
=&H(\mathbb{X}_\mu, \mathbb{X}_\nu) +  2 \int_{u>0} \mathbb{P}\left(
H(X,\mathbb{X}_\mu) \geq u \right) du \\
=&H(\mathbb{X}_\mu, \mathbb{X}_\nu) +  2 \mathbb{E}\left[ H( X, \mathbb{X}_\mu )
 \right] \\
\leq & H(\mathbb{X}_\mu, \mathbb{X}_\nu) + 2r_0 \, + \, 4\left(\frac{\log m}{am}
\right)^{1/b} \mathbbm{1}_{(r_0,\infty)}( r_m) \, + \, 4 \, C_1(a,b) \left(\frac{\log m}{am} \right)^{1/b} \frac{1}{(\log
m)^2},
\end{align*}
where the last inequality follows from Lemma \ref{lem:covering}.
\qed

\paragraph{Proof of Theorem \ref{th:riskAverage}}
It follows directly from \eqref{eq:hausdorff} and Lemma \ref{lem:covering} . 
\qed

\paragraph{Proof of Theorem \ref{theo:RiskClosest}}
\begin{align*}
\mathbb{E}\left[ \Vert \lambda_{\mathbb{X}_\mu} - \hat{\lambda_n^m}
\Vert_{\infty} \right] & \leq \mathbb{E}\left[ H(\mathbb{X}_\mu, \hat{C^m_n})
\right] \\
& \leq \int_{r>0} \mathbb{P}\left( H(\mathbb{X}_\mu, \hat{C^m_n}) > r \right) dr
\\
& \leq r_0 + \int_{r>r_0} \left[\mathbb{P}\left( H(\mathbb{X}_\mu, {S_1^m}) > r
\right) \right]^n dr \\
& \leq r_0 + \int_{r>r_0} \left[  \frac{4^b}{a r^b}    \exp\left( -   m   \frac{
a}{2^b}  r^b\right) \right]^n dr ,
\end{align*}
where the last inequality follows from Lemma \ref{lem:covering}.
If $r_m \leq r_0$ then the last term is upper bounded by
$$
r_0 +  \int_{r>r_m} \left[  \frac{4^b}{a r^b}    \exp\left( -   m   \frac{
a}{2^b}  r^b\right) \right]^n dr,
$$
otherwise it is bounded by
$$
r_0 + r_m +  \int_{r>r_m} \left[  \frac{4^b}{a r^b}    \exp\left( -   m   \frac{
a}{2^b}  r^b\right) \right]^n dr. \\
$$
In either case,
\begin{align*}
\int_{r>r_m} \left[  \frac{4^b}{a r^b}    \exp\left( -   m   \frac{ a}{2^b}
r^b\right) \right]^n dr &= 2\frac{ 2^{bn}}{b} (ma)^{-1/b} m^n \int_{u>\log m}
u^{1/b-n-1} \exp(-n u) du \\
&\leq 2 C_2(a,b) \left(\frac{\log(2^b m)}{am} \right)^{1/b} \frac{1}{n \; [ \,
\log(2^b m)]^{n+1}} \; ,
\end{align*}
where in the last inequality we applied the same strategy used to prove Theorem 2 in \cite{chazal2013optimal}.
\qed

\section{About the $(a,b,r_0)$-standard assumption }
\label{sec:abtrunc}

The aim of this section is to explain why the $(a,b, r_0)$-standard assumption  is relevant, in particular when $\mu$ is a discrete measure. Our argument is related to weighted empirical processes, which have been studied in details by Alexander; see \cite{Alexander85,alexander1987rates,alexander1987central}. A new look on this problem has been proposed more recently in \cite{GineKolch06,GKW03} by using Talagrand concentration inequalities. The following result  from \cite{Alexander85} will be sufficient here.
Let $(\mathbb{X}, \rho,\eta)$ be a measure metric space and let $\eta_N$ be the empirical counterpart of $\eta$. 
\begin{proposition}
Let $\mathcal C$ be a VC class of measurable sets of index $v$ of $\mathbb{X}$. Then for every $\delta$,
$\varepsilon >0$ there exists $K$ such that
\begin{equation} \label{eq:Alex}
\eta \left[
\sup \left\{
 \left|  \frac{\eta_N(C)-\eta(C)}{\eta(C)}   \right|  \: : \:
   \eta(C)  \geq K v \frac {\log N} N, \, C \in \mathcal C
 \right\}    > \varepsilon
\right]
 = O(N ^{-(1+ \delta)v} ) .
\end{equation}
\end{proposition}

Assume that $\mu$ is the discrete uniform measure on a point cloud $X_N=\{x_1, \dots, x_N\}$ which has been sampled from  $\eta$, thus $\mu=\eta_N$. Assume moreover that $\eta$ satisfies an $(a',b)$-standard assumption ($r_0 =0$). Let  $r_0$ be a positive function of $N$  chosen further.  For any $r>r_0(N) $ and any $y\in   \mathbb{X}_\mu$:
\begin{eqnarray}
  \inf_{y \in \mathbb X _\mu} \mu (B(y,r))  &= &   \inf_{y \in \mathbb X_\mu} \eta_N(B(y,r))  \nonumber  \\
 &= &    \inf_{y \in \mathbb X _\mu}  \left\{ \eta(B(y,r))  \left[ 1 -  \frac{ \eta(B(y,r))  -  \eta_N(B(y,r)) }{  \eta(B(y,r))} \right]  \right\} \nonumber \\
 &\geq& ( 1 \wedge a' r^b )   \inf_{y \in \mathbb X _\mu}  \left\{ 1   -  \sup _{x \in  \mathbb{X}} \left|  \frac{ \eta(B(x,r)) - \eta_N (B(x,r)) }{\eta(B(x,r))}  \right| \right\}  \nonumber   \\
 &\geq& ( 1 \wedge a' r^b )   \inf_{y \in \mathbb X _\mu}  \left\{ 1   -  \sup _{x \in  \mathbb{X}, r' \geq r_0(N)} \left|  \frac{ \eta(B(x,r')) - \eta_N (B(x,r')) }{\eta(B(x,r'))}  \right| \right\}  \label{ineqabrN}
\end{eqnarray}
Assume that the set of balls in $(\mathbb{X}, \rho)$ has a finite VC-dimension $v$. For instance, in $\R^d$, the VC-dimension of balls is d+1. Under this assumption we apply Alexander's Proposition with (for instance) $\delta =1 $ and $\varepsilon = 1/2$.  Let $K>0$ such that (\ref{eq:Alex}) is satisfied. Then, by setting
$$r_0(N) :=  \left( \frac{ Kv }{a'}\frac {\log N}  N \right)^{1/b} \, ,$$  
we finally obtain using (\ref{eq:Alex}) and (\ref{ineqabrN}) that
$$ \eta \left[
 \inf_{y \in \mathbb X _\mu, \, r \geq r_N} \mu (B(y,r)) \geq 
 1 \wedge \frac {a'} 2 r^b     
\right]
 = O(N ^{-2 v} ) .$$
In this quite  general context, we see that  by taking $r_0 $ of the order of $\left( \frac {\log N}  N \right)^{1/b}$, for large  values of $N$ the $(a,b,r_0)$-standard assumption is satisfied with high probability (in $\eta$).

\section{Robustness to Outliers}
\label{sec:robustness}

The average landscape method is insensitive to outliers, as can be seen by the stability result of Theorem \ref{thm:stabW}.
The probability mass of an outlier gives a minimal contribution to the Wasserstein distance on the
right hand side of the inequality.\\
For example, suppose that $X_N=\{x_1, \dots, x_N \}$ is a random sample from the
unit circle $\mathbb{S}^2$, and let
$Y_N=X_N\backslash \{x_1\} \cup \{(0,0)\} $. See Figure \ref{fig:exCircle}.  The
landscapes $\lambda_{{X_N}}$ and $\lambda_{{Y_N}}$ are very different
because of the presence of the outlier $(0,0)$. On the other hand, the average
landscapes constructed by multiple subsamples of $m<N$ points from $X_N$ and
$Y_N$ are close to each other. Formally, let $\mu$ be the discrete uniform
measure that put mass $1/N$ on each point of $X_N$ and similarly let $\nu$ be
the discrete uniform measure on $Y_N$. 
The 1st Wasserstein distance between the
two measure is 1/N and, according to Theorem \ref{thm:stabW},  
the difference between the average landscapes is
$
\left\Vert \mathbb{E}_{\Psi_{\mu}^m}[\lambda_X] -
\mathbb{E}_{\Psi_{\nu}^m}[\lambda_Y] \right\Vert_\infty \leq  \frac{m}{N}.
$

\begin{figure}[!ht]
\centering
\includegraphics[width=5.4in]{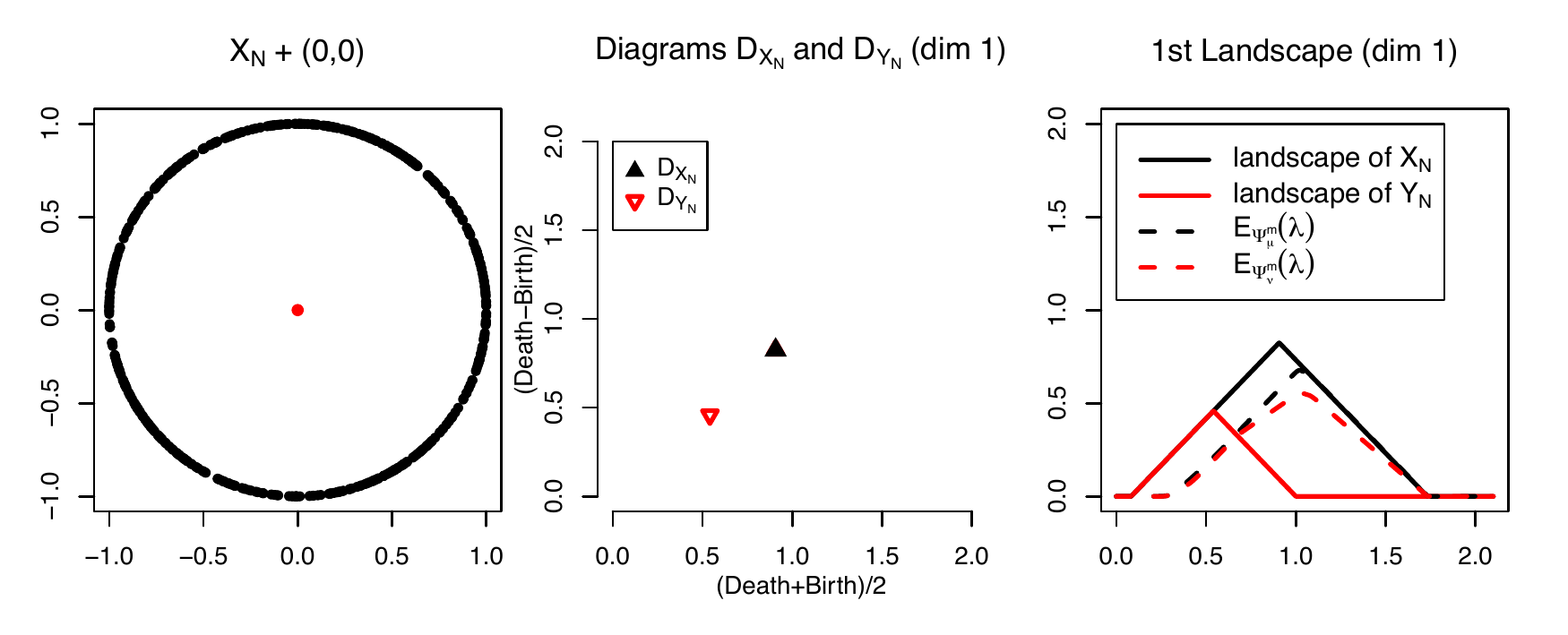}
\caption{Left: $X_N$ is the set of $N=500$ points on the unit circle; $Y_N=X_N\backslash \{x_1\} \cup \{(0,0)\}$. Middle: persistence diagrams (dim 1) of the VR filtrations on $X_N$ and $Y_N$, in the same plot, with different symbols. Right: landscapes of $X_N$, $Y_N$ and the corresponding average landscapes constructed by subsampling $m=100$ points from the two sets, for $n=30$ times. }
\label{fig:exCircle}
\end{figure}

More formally, we can show that the average landscape
$\overline{\lambda_n^m}$ can be more accurate
than the closest subsample method
when there are outliers.
In fact,
$\overline{\lambda_n^m}$  can even be more accurate
than the landscape corresponding to a large sample of $N$ points.

Suppose that the large, given point cloud $X_N=\{x_1, \dots, x_N \}$ 
has a small fraction of outliers.
Specifically,
$X_N = {\cal G} \bigcup {\cal B}$
where
${\cal G} = \{x_1,\ldots, x_G\}$
are the good observations
and ${\cal B} = \{y_1,\ldots, y_B\}$
are the outliers (bad observations).
Let $G = | {\cal G}|$,
$B = | {\cal B}|$ so that $N = G + B$
and let
$\epsilon = B/N$
which we assume is small but positive.
Our target is the landscape based on the non-outliers,
namely,
$\lambda_{\mathcal{G}}$.
The presence of the outliers means that
$\lambda_{X_N} \neq \lambda_G$.
Let 
$
\beta = \inf_S ||\lambda_S - \lambda_{\mathcal{G}}||_\infty > \delta,
$
for some $\delta>0$,
where the infimum is over all subsets that contain at least one outlier.
Thus, $\beta$
denotes the minimal bias due to the outliers.
We consider three estimators:
\begin{align*}
\lambda_{X_N} &: {\rm landscape\ from\ full\ given\ sample\ } X_N\\
\overline{\lambda_n^m} &: {\rm average \ landscape\ from\ } n {\rm \ subsamples\ of\ size\ } m\\
\hat{\lambda_n^m}  &: {\rm landscape\ of\ closest\  subsample,\  from\ } n {\rm \ subsamples\ of\ size\ } m.
\end{align*}

The last two estimators are defined in Section 3, and are constructed using $n$ independent
samples of size $m$ from the discrete uniform measure that puts mass $1/N$ on each point of $X_N$.

\begin{proposition}
If $\epsilon = o(1/n)$,
%%%\begin{equation}\label{eq::cond}
%%%\frac{1}{2} + e^{-m\epsilon} >
%%%\sqrt{\frac{\log n}{4m}}
%%%\end{equation}
then, for large enough $n$ and $m$,
\begin{equation}
\mathbb{E}\Vert \overline{\lambda_n^m} - \lambda_{\mathcal{G}}\Vert_\infty <
\mathbb{E}\Vert\lambda_{X_N} - \lambda_{\mathcal{G}} \Vert_\infty.
\end{equation}
In addition, if
$n m \epsilon\to \infty$ then
\begin{equation}
\mathbb{P}\left[ \mathbb{E}\Vert \overline{\lambda_n^m}- \lambda_{\mathcal{G}}\Vert_\infty < 
\Vert \hat{\lambda_n^m} -\lambda_{\mathcal{G}}\Vert_\infty \right] \to 1.
\end{equation}
\end{proposition}

\begin{proof}
Say that a subsample is clean if it contains no outliers
and that it is infected if it contains at least one outlier.
Let $S_1,\ldots,S_n$ be the subsamples of $X_N$ of size $m$.
Let $I = \{i:\ S_i\ {\rm is\ infected}\}$ and
$C = \{i:\ S_i\ {\rm is\ clean}\}$.
Then
$$
\overline{\lambda_n^m} = \frac{n_0}{n}\lambda_0 +  \frac{n_1}{n}\lambda_1
$$
where
$n_0$ is the number of clean subsamples,
$n_1 = n-n_0$
is the number of infected subsamples,
$\lambda_0 = (1/n_0)\sum_{i\in C}\lambda_{S_i}$ and
$\lambda_1 = (1/n_1)\sum_{i\in I}\lambda_{S_i}$.
Hence,
\begin{align*}
\Vert \overline{\lambda_n^m} - \lambda_{\mathcal{G}} \Vert_\infty &\leq
\frac{n_0}{n} \Vert\lambda_0 -\lambda_{\mathcal{G}} \Vert_\infty + \frac{n_1}{n} \Vert\lambda_1
-\lambda_{\mathcal{G}}\Vert_\infty\\
& \leq
\frac{n_0}{n}\Vert\lambda_0 -\lambda_{\mathcal{G}}\Vert_\infty + \frac{T n_1}{2n}.
\end{align*}

A subsample is clean with
probability $(1-\epsilon)^m$.
Thus,
$n_0 \sim {\rm Binomial}(n,(1-\epsilon)^m)$ and
$n_1 \sim {\rm Binomial}(n,1-(1-\epsilon)^m)$.
Let $\pi = 1-(1-\epsilon)^m$.
By Hoeffding's inequality
$$
\mathbb{P}\left(\frac{T n_1}{2n} > \frac{\beta}{2}\right) =
\mathbb{P}\left(\frac{T n_1}{2n}-\frac{T\pi}{2} > \frac{\beta}{2}-\frac{T\pi}{2}\right)
 \leq
\exp\left( - 2n \left(\frac{\beta}{T}-\pi\right)^2\right).
$$
Since $\epsilon = o(1/n)$,
we eventually have that
$$
\pi=1 - (1-\epsilon)^m < \frac{\beta}{T} - \sqrt{\frac{\log n}{2n}},
$$
which implies that
$\mathbb{P}\left(\frac{T n_1}{2n} > \frac{\beta}{2}\right)< 1/n$.
So, except on a set of probability tending to 0,
$$
\Vert \overline{\lambda_n^m} - \lambda_{\mathcal{G}} \Vert_\infty \leq
\frac{n_0}{n}\Vert\lambda_0 -\lambda_{\mathcal{G}}\Vert_\infty + \frac{\beta}{2} \leq
\Vert\lambda_0 -\lambda_{\mathcal{G}}\Vert_\infty + \frac{\beta}{2}
$$
and thus, as soon as $n, m$ and $N$ are large enough,
$$
\mathbb{E}\Vert \overline{\lambda_n^m}  - \lambda_{\mathcal{G}} \Vert_\infty \leq
\frac{\beta}{2} + \frac{\beta}{2} =
\beta \leq \Vert\lambda_{X_N} - \lambda_{\mathcal{G}}\Vert_\infty.
$$
This proves the first claim.
To prove the second claim, note that
the probability that at least one subsample is infected is
$1-(1-\epsilon)^{nm} \sim 1 - e^{-\epsilon nm} \to 1$.
So with probabilty tending to 1,
there will be an infected subsample.
This subsample will minimize
$H(X,S_j)$ and the landscape based on this selected
subsample will have a bias of order $\beta$.
\end{proof}

In practice, we can increase the robustness further, by
using filtered subsampling.
This can be done using distance to $k$-th nearest neighbor or
using a kernel density estimator.
For example,
let
$$
\hat p_h(x) = \frac{1}{N} \sum_{j=1}^N K\left(\frac{||x-X_i||}{h}\right)
$$
be a kernel density estimator
with bandwidth $h$ and kernel $K$.
Suppose that all subsamples are chosen from the filtered set
${\cal F} = \{X_i:\ \hat p_h(X_i) > t\}$.
Suppose that the good observations ${\cal G}$ are sampled from a
distribution on a set $A\subset [0,1]^d $ satisfying the $(a,b)$-standard condition
with $b < d$, $a>0$
and that ${\cal B}$ consists of $B$ observations sampled from a uniform
on $[0,1]^d$.
For any $x\in A$,
$$
\mathbb{E}[\hat p_h(x)] \approx \frac{a h^b}{h^d}
$$
and for any outlier $X_i$ we have
(for $h$ small enough) that
$\hat p_h(X_i) = 1/(nh^d)$.
Hence, if we choose $t$ such that
$$
\frac{1}{nh^d} < t < \frac{a }{h^{d-b}}
$$
then
${\cal F}={\cal G}$
with high probability.

\end{document}